\documentclass{amsart}
\usepackage{amsmath,amssymb,amsfonts}
\usepackage{amsthm}
\usepackage{enumerate}
\usepackage{young}
\usepackage{xcolor}
\usepackage{graphicx}
\usepackage{float}
\usepackage{mathtools}
\usepackage{hyperref}
\usepackage{enumitem}

\usepackage{tikz}
\usetikzlibrary{arrows}
\usepackage{mathtools}
\newcommand{\defeq}{\vcentcolon=}

\newtheorem{theorem}{Theorem}[section]
\newtheorem{corollary}[theorem]{Corollary}
\newtheorem{lemma}[theorem]{Lemma}

\newtheorem{proposition}[theorem]{Proposition}

\newtheorem{conjecture}[theorem]{Conjecture}
\newtheorem{definition}[theorem]{Definition}
\newtheorem{example}[theorem]{Example}

\newcommand\inv{\operatorname{Inv}}

\DeclareMathOperator{\sh}{sh}
\newcommand{\transp}[1]{s_{#1}}
\newcommand{\inner}{\overline{\mu}}
\renewcommand{\outer}{\overline{\nu}}

\begin{document}

\title{The number of inversions of permutations with fixed shape}
\date{\today}
\author{Arvind Ayyer} 
\address{Arvind Ayyer\\
Department of Mathematics, Indian Institute of Science\\ Bangalore 560012, India}
\email{arvind@iisc.ac.in}

\author{Naya Banerjee}
\address{Naya Banerjee\\
Department of Mathematical Sciences, University of Delaware\\
Newark, DE, 19716, USA
}
\email{naya@udel.edu}

\subjclass[2010]{05A05; 05A15; 05A17}
\keywords{Robinson-Schensted correspondence; inversions; minimal permutations; jump partitions}

\begin{abstract}
The Robinson-Schensted correspondence naturally induces a map from permutations to partitions.
In this work, we study the number of inversions of permutations corresponding to a fixed partition $\lambda$ under this map. 
Hohlweg characterized permutations having shape $\lambda$ with the minimum number of inversions. 
Here, we give the first results in this direction for higher numbers of inversions. 
We give explicit conjectures for both the structure and the number of permutations associated to $\lambda$ where the extra number of inversions is less than the length of the smallest column of $\lambda$.
We prove the result when $\lambda$ has two columns.
\end{abstract}

\maketitle

\section{Introduction}
The Robinson-Schensted (RS) correspondence is a remarkable bijection mapping permutations in $S_n$ to pairs of standard Young tableaux of the same shape. 
It was discovered independently by Robinson \cite{Robinson38} and Schensted
\cite{Schensted61}. This bijection is intimately related to the 
representation theory of the symmetric group \cite{JamesKerber81,Dia88}, the theory of 
symmetric functions \cite[Chapter 7]{Stanley1999}, and the theory of
partitions \cite{Andrews76}.
Using this correspondence, one can associate a partition to any permutation by looking at the shape of either of the tableaux. It is natural to ask what the relationship of this partition statistic is to other well-studied statistics on permutations. One well-studied statistic on permutations is the inversion number, i.e., the number of pairs of elements in the permutation that are out of the natural order.
Although it is natural to study the relationship between the inversion number and the shape of the partition corresponding to a permutation, this seems to be a difficult problem (see \cite{Reifegerste04,MuellerStarr2013}, for instance),
and there has been progress in essentially one special case. 
As of this writing, the only results were due to
Hohlweg \cite{Hohlweg2005}, and there had been no further progress on the question. Hohlweg used the theory of Kazhdan-Lusztig cells to 
determine the minimum number of inversions of permutations associated to a shape. Further, he gave a complete characterization for such minimal permutations. 
Subsequently, Han \cite{Han2005} gave a combinatorial proof of this result. 
Hohlweg also showed that these minimal permutations with shape $\lambda$ are in bijection with permutations with the maximum number of inversions in the conjugate partition $\lambda'$.

Our main result adapts Han's proof and extends Hohlweg's characterization for minimal elements associated to partitions to elements with larger numbers of inversions in the case of two-column tableaux when the number of inversions is bounded. Although our result applies to the special case of Young diagrams with two columns, we note that this is the first development on the problem in over a decade. We conjecture that our characterization does generalize to more than two columns, although this seems much more difficult to prove using the approach in this work.

The remainder of the paper is organized as follows. In Section \ref{sec:preliminaries} we summarize the relevant background and state our results. In Section \ref{sec:jumppart} we introduce a combinatorial construction we call a \textit{jump partition} that we will use to characterize the minimal permutations and their generalizations. Finally, in Section \ref{sec:two-col}, we will prove the main result.

\section{Preliminaries and Summary of Results}
\label{sec:preliminaries}

Let $n \in \mathbb N$ be a positive integer. A vector
$\lambda=(\lambda_1, \lambda_2, \dots)$ of positive integers  is a
{\em partition} of $n$ (denoted by $\lambda \vdash n$)
if
\[
\lambda_1 \ge \lambda_2 \ge \dots > 0 \ \mathrm{and}
\ \sum_{i}\lambda_i = n.
\]
The size of a partition $\lambda$ will be denoted by $|\lambda|$ and the number of non-zero parts in $\lambda$, by $\ell(\lambda)$. 
The {\em Young diagram} (or diagram) of a partition $\lambda$ is a left-justified
array of cells with $\lambda_i$ cells in the $i$-th row for each
$i\geq 1$.
For example, the diagram of the partition $(5,5,3,2)$ is
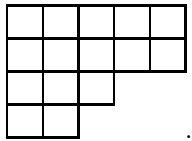
\[
\begin{Young}
& & & &\cr
& & & &\cr
& &  \cr
& \cr
\end{Young}.
\]
The {\em conjugate} of a partition $\lambda$, denoted by $\lambda'$, is
the partition whose diagram is the transpose of the diagram of $\lambda$.
A {\em standard Young tableau} ({\em SYT} or {\em tableau}) of shape $\lambda$ with
is a filling of the Young diagram of $\lambda \vdash n$ with entries from $[n] := \{1,\dots,n\}$ in such a
way that the entries are strictly increasing from left to right along every
row as well as from top to bottom along every column. The {\em shape} of a
tableau $T$, denoted $\sh(T)$ is the partition corresponding to the
diagram of $T$. For example, 
\[
\begin{Young}
1& 2 & 4  & 7 & 8 \cr
3& 6 & 10 & 12 & 13\cr
5 &9 & 14 \cr
11 & 15\cr
\end{Young}
\]
is a tableau of the diagram above. Note that the elements in
the cells of a SYT are distinct integers. 
Let
$\mathcal T_n$ denote the set of SYT of size $n$.

\subsection{The Robinson-Schensted Correspondence}
The Robinson-Schensted or RS correspondence \cite{Robinson38, Schensted61} is a 
bijection between the set of permutations $S_n$ and the set of pairs of tableau of size $n$
of the same shape. 
The definition of the bijection is through a {\em row-insertion} algorithm
first defined by Schensted \cite{Schensted61} in order to study the longest
increasing subsequence of a permutation. 
We first review the bijection.
Suppose that we have a tableau $T$. The row-insertion procedure below inserts a positive integer $x$
that is distinct from all entries of $T$, into
$T$ and results in a tableau denoted by $T \leftarrow x$.
\begin{enumerate}
\item Let $y$ be the smallest number larger than $x$ in the first row of
  $T$. Replace the cell containing $y$ with $x$. If there is no such
  $y$, add a cell containing $x$ to the end of the row.
\item If $y$ was removed from the first row, attempt to insert it into
  the next row by the same procedure as above. If there is no row to
  add $y$ to, create a new row with a cell containing $y$.
\item Repeat this procedure on successive rows until either a number
  is added to the end of a row or added in a new row at the bottom.
\end{enumerate}

The RS correspondence from $S_n$ to $\{(P,Q) \in \mathcal T_n
\times \mathcal T_n\ : \ \sh(P)=\sh(Q)\}$ can now be defined as
follows. Let $\pi \in S_n$ be written in one-line notation as
$(\pi_1,\dots,\pi_n)$. Let $P_1$ be the tableau with a single cell containing
$\pi_1$. Let $P_{j} = P_{j-1} \leftarrow \pi_j$ for all $1<j \le n$ and
set $P=P_n$. The tableau $Q$ is defined recursively in terms of tableaux $Q_i$
of size $i$ as follows. Let $Q_1$ be
the tableau with one cell containing the integer $1$. The equality of shapes 
$\sh(Q_i)=\sh(P_i)$ is maintained throughout the process. The cell of $Q_i$ containing $i$ is the
(unique) cell of $P_i$ that does not belong to $P_{i-1}$. The remaining cells of
$Q_i$ are identical to those of $Q_{i-1}$. Finally, set
$Q=Q_n$. We refer to $P$ as the {\em insertion tableau} and $Q$ is
the {\em recording tableau}.

Let $\pi \in S_n$ and let $(P,Q)$ be the corresponding tableaux under
the RS correspondence. The {\em shape of $\pi$}, denoted $\sh(\pi)$, is given by $\sh(P)=\sh(Q)$.
We will also say that a partition $\lambda$ is {\em associated to} a permutation $\pi$ if 
$\lambda = \sh(\pi)$. 
A fundamental property about the RS
correspondence is that if $\pi$ is associated to $(P,Q)$, then $\pi^{-1}$
is associated to $(Q,P)$.
See \cite{Fulton97,Mac98,Knu73} or \cite{Stanley1999} for numerous interesting properties of the bijection.

\subsection{Main Results}
Let $P_c(n)$ denote the set of partitions of $n$
into parts of $c$ colors or types, and let $p_{c}(n) = \# P_{c}(n)$.
For example, $P_2(2) = \{ 2, \bar{2}, 11, 1 \bar{1}, \bar{1} \bar{1} \}$, 
where we think of the two colors as unbarred and barred integers.
Further, let $p(n) \equiv p_{1}(n)$ denote the number of
partitions of $n$. 
{Since any $c$-colored partition of $n$ can be obtained by writing $n$
as a sum of nonnegative integers $n_1 + \cdots + n_c$ and forming a partition of $n_i$ with the $i$'th colour, we have the recursion
\begin{equation}
\label{pc-recur}
p_c(n) = \sum_{n_1 + \cdots + n_c = n} p(n_1) \cdots p(n_c).
\end{equation}
Therefore, the generating function of $p_{c}(n)$ is given by the 
$c$-fold convolution
\begin{align}
\sum_{n \geq 0} p_{c}(n) x^{n} = \prod_{m \geq 1}\frac 1{(1-x^{m})^{c}}. 
\end{align}
See, for example ~\cite[Sequence A000712]{OEIS}  for $c=2$.
}

Let $\lambda$ denote the partition of $n$ such that its conjugate is given in the frequency notation as $\lambda' = \langle t_1^{m_1}, \dots, t_k^{m_k} \rangle$, where $t_1>\cdots > t_k>0$ and $m_i>0$ is the multiplicity of $t_i$. In other words, $\lambda'$ has the part $t_1$ occurring $m_1$ times, $t_2$ occurring $m_2$ times, and so on.
Therefore $n = \sum_{i=1}^k t_i m_i$. We will also frequently need the sum of the multiplicities, which we will denote $M = \ell(\lambda') = \sum_{i=1}^k m_i$. 

We say that a permutation associated to $\lambda$ is {\em minimal} for the shape $\lambda$ if it has the minimal number of inversions.
Let $W_\lambda$ be the set of all permutations with shape $\lambda$ and $W_\lambda^\Delta$ to be the set of permutations associated to $\lambda$ with exactly $\Delta$ more inversions than the minimal one so that
$\cup_{\Delta \geq 0} W_\lambda^\Delta = W_\lambda$. 
Let $w_\lambda^\Delta = \# W_\lambda^\Delta$. 
Hohlweg has proved the following results for minimal permutations.

\begin{theorem}[{\cite{Hohlweg2005}}]
\label{thm:holhweg}
Let $\lambda$ be the shape given by $\lambda'=\langle t_1^{m_1}, \dots, t_k^{m_k} \rangle = (\lambda'_1,\dots,\lambda'_M)$. Then the following results hold for minimal permutations with shape $\lambda$.
\begin{enumerate}
\item The number of inversions of such a minimal permutation is 
\[
\sum_{i=1}^M \binom {\lambda'_i}2.
\]

\item The number of minimal permutations is given by
\[
w_\lambda^0 =  \binom M {m_1,\dots,m_k}.
\]

\item The minimal permutations can be characterized as follows.
Let $c=(c_1,\ldots,\allowbreak c_M)$ be a permutation (possibly with repeats) of $(\lambda'_1,\dots,\lambda'_M)$. In one-line notation, the corresponding minimal permutation (see the example after this theorem) is composed of $M$ blocks of entries from $[n]$ of lengths $c_1,\ldots,c_M$ in that order, such that the elements within each block are consecutively decreasing and the elements of the $i$'th block are smaller than the elements of the $(i+1)$'th block for each $1 \le i <M$.

\end{enumerate}
\end{theorem}

Minimal permutations are also known as {\em layered permutations}~\cite{Bona1999}. 
For example, $(2,1 \,|\, 3\, |\, 6,5,4\, |\, 8,7)$ is a minimal permutation for the shape $\lambda = (4,3,1)$
corresponding to the permutation $c = (2,1,3,2)$ of $\lambda' = (3,2,2,1)$. 
We place dividing bars between various blocks for ease of reading.
{As can be seen from the characterization in \textit{(3)} above,} minimal permutations are involutions. 

We are interested in the case where $\Delta$ is positive.
Let us first look at the case when $M=1$. Then $k=m_1 =1$ and $t_1 = n$ and $\lambda' = \langle 1^n \rangle$. Since $\lambda = (n)$ consists of a single row, there is exactly one permutation which has that shape by the RS correspondence, namely the identity permutation with inversion number 0. Therefore 
\[
w_{(n)}^\Delta = 
\begin{cases}
1, & \text{if } \Delta = 0, \\
0, & \text{if } \Delta > 0.
\end{cases}
\]
The first nontrivial result will be when $M=2$. Assume the notation above.

\begin{theorem}\label{thm:two-col}
Let $\lambda \vdash n$ with $\lambda' = (s,r)$, and let $\Delta < r$. Then,
\[
w^{\Delta}_\lambda = p_{2} (\Delta) \times
\begin{cases}
1, & \mathrm{if} \; \ s = r, \\
2, & \mathrm{if} \; \ s > r.
\end{cases}
\]
\end{theorem}

We will prove Theorem~\ref{thm:two-col} in Section~\ref{sec:two-col}.
Our proof does not generalize to $M>2$, although we conjecture a more general result from empirical data. We make partial progress in that direction with the following lower bound.

\begin{theorem}
\label{thm:more than two col}
Fix a partition $\lambda$ such that $\lambda'=\langle t_1^{m_1}, \dots, t_k^{m_k} \rangle$ with $t_1 > \cdots > t_k > 0$ and $m_i>0$ for all $i$.
Denote $M = \sum_{i=1}^k m_i$. Then for any $\Delta< t_k$,
\begin{align*}
w^{\Delta}_\lambda \geq p_{2(M-1)} (\Delta) \binom M {m_1,\dots,m_k}.
\end{align*}
\end{theorem}

We will prove Theorem~\ref{thm:more than two col} in Section~\ref{sec:lowerbd}. We believe that the following stronger conjecture holds.

\begin{conjecture}
\label{conj:more than two col}
Fix a partition $\lambda$ such that $\lambda'=\langle t_1^{m_1}, \dots, t_k^{m_k} \rangle$ with $t_1 > \cdots > t_k > 0$ and $m_i>0$ for all $i$.
Denote $M = \sum_{i=1}^k m_i$. Then for any $\Delta< t_k$,
\begin{align*}
w^{\Delta}_\lambda = p_{2(M-1)} (\Delta) \; \binom M {m_1,\dots,m_k}.
\end{align*}
\end{conjecture}

We have verified Conjecture~\ref{conj:more than two col} for partitions of size up to 10.
The methods in this work do not seem to work when the shape of the permutation has more than two columns. We also note that we do not have any results of this nature when $\Delta$ is greater than or equal to $t_k$. It would be of great interest to generalize our results and conjectures to that case.

\section{Jump partitions and Knuth equivalence}
\label{sec:jumppart}
We are interested in counting the number of
permutations with a given number of inversions associated to the shape $\lambda$
via the RS correspondence. 
A natural and very useful tool is {\em Knuth equivalence} of permutations
\cite{Knuth1970}. A very readable introduction is by S. Fomin in \cite[Appendix 1]{Stanley1999}. A {\em Knuth transformation} of a permutation transforms adjacent entries $a,b,c$
satisfying $a<b<c$ in one of the following ways: $\cdots acb \cdots
\leftrightarrow \cdots cab \cdots$ or $\cdots bac \cdots
\leftrightarrow \cdots bca \cdots$. For our purposes, we will also need to keep track of the inversion number. We will therefore use the notation
\[
\text{$K_+$ moves to mean transpositions}
\begin{cases}
\cdots acb \cdots \rightarrow \cdots cab \cdots \\
\cdots bac \cdots \rightarrow \cdots bca \cdots
\end{cases}
\]
which increase the inversion number, and
\[
\text{$K_-$ moves to mean transpositions}
\begin{cases}
\cdots cab \cdots \rightarrow \cdots acb \cdots \\
\cdots bca \cdots \rightarrow \cdots bac \cdots
\end{cases}
\]
which decrease the inversion number.
Two permutations are {\em Knuth equivalent} if one can be obtained
from another using Knuth transformations.  
A {\em dual Knuth transformation} of a permutation is a Knuth
transformation of its inverse.
By the RS correspondence, two dual
Knuth equivalent permutations have the same $Q$-tableau.
We also need a more explicit notation for these.
We use
\[
\text{$KD_+$ moves to mean}
\begin{cases}
\cdots a \cdots (a-1) \cdots (a+1) \cdots \rightarrow
\cdots (a+1) \cdots (a-1) \cdots a\\
\cdots a \cdots (a+2) \cdots (a+1) \cdots \rightarrow
\cdots (a+1) \cdots (a+2) \cdots a\\
\end{cases}
\]
which increase the inversion number, and
\[
\text{$KD_-$ moves to mean}
\begin{cases}
\cdots (a+1) \cdots (a-1) \cdots a \cdots \rightarrow
\cdots a \cdots (a-1) \cdots (a+1)\\
\cdots (a+1) \cdots (a+2) \cdots a \cdots \rightarrow
\cdots a \cdots (a+2) \cdots (a+1)
\end{cases}
\]
which decrease the inversion number.  Note that $K_+$ and $K_-$ are
right actions of elementary transpositions and $KD_+$ and $KD_-$ are
left actions.  The importance of these moves for us comes from the following result, which is one of the important properties of Knuth moves.

\begin{proposition}[{\cite{Knuth1970}}]
Two permutations are associated to tableaux of the same shape via the
RS correspondence if and only if one can be obtained from the other by a series
of Knuth and dual Knuth tranformations.
\end{proposition}

As usual, let $\transp{i}$ for $i=1,\dots,n-1$ denote the standard generators
of the symmetric group $S_n$ by adjacent transpositions.  The left (resp. right) action of
$\transp{i}$ interchanges the values (resp. entries at positions) $i$
and $i+1$. 
To simplify notation, fix a minimal permutation $\pi$ corresponding to the composition $c=(c_1,\ldots,c_M)$, let 
$c_j^* = \sum_{i=1}^j c_i$ and define $c_0^*$ to be $0$.
{
Recall that a \emph{partition} $\lambda = (\lambda_1,\dots,\lambda_k)$ is a weakly decreasing sequence of positive integers. The \emph{size} of $\lambda$ is given by $|\lambda| = \sum_{i=1}^k \lambda_i$.
}

\begin{definition}
\label{def:inner}
Given a minimal permutation $\pi$ with $M$ blocks, an {\bf inner jump partition $\inner$ of size $\Delta$} for $\pi$ is a
sequence of partitions $\inner=(\mu^{(1)},\dots,\mu^{(M-1)})$,
satisfying the following conditions.
\begin{enumerate}
\item $|\inner| := \sum_{i=1}^{M-1} |\mu^{(i)}| = \Delta$, and
\item for $1 \le i \le  M-1$, $\ell(\mu^{(i)}) \le c_i$ and $\mu^{(i)}_1 < c_{i+1}$.
\end{enumerate}

An inner jump partition $\inner$ {\em acts} on a minimal permutation $\pi$ by right multiplication of transpositions $\transp{i}$. In particular, define the action of $\inner$  as
\begin{align}
\label{inner-action}
\pi \circ \inner & \defeq  \; \pi \cdot \prod_{i=1}^{{M-1}} \left( \transp{c_i^*} \cdots \transp{c_i^*-1+\mu^{(i)}_1} \right) \cdots \left(\transp{c_i^*-\ell(\mu^{(i)})+1} \cdots  \transp{c_i^*-\ell(\mu^{(i)})+\mu^{(i)}_{\ell(\mu^{(i)})}} \right),
\end{align}
{where the action of the product is in increasing order.}
We will call the action of each transposition above as an {\em inner jump}.
\end{definition}

\begin{example}
\label{eg:inneronmin}
Let $\pi = (3,2,1\, |\, 7,6,5,4\, |\, 10,9,8)$ be a minimal permutation for the partition $\lambda = (3,3,3,1)$ {with $c = (3,4,3)$}, and let $\inner = (\emptyset,(1,1))$. Then,
{$c^*_1 = 3$ and $c^*_2 = 7$ and therefore}
\[
\pi \circ \inner = \pi \cdot 
{
\underbrace{(\text{id})}_{\mu_1 = \emptyset} 
\cdot
\underbrace{\transp{7} \cdot \transp{6}}_{\mu_2 = (1,1)}
}
= (3,2,1\, |\, 7,6,10,5\, |\, 4,9,8).
\]
{For the sake of clarity, we have kept the locations of the dividing bars unchanged and will do so in the examples that follow as well. From this example, the reason for naming the action of the transposition an ``inner jump" becomes clearer (see also Example \ref{eg:M=2}). If one assigns the parts of the partition $(1,1)$ to the elements $5$ and $4$, the size of the part is the number of ``jumps" performed by the element to the right, with the jumps being performed by the elements in the second block over elements of the third block, across the dividing line - the ``inner" boundary - between the second and third blocks (in contrast to the ``outer" boundary which we describe next).}
\end{example}

\begin{definition}
\label{def:outer}
Given a minimal permutation $\pi$ with $M$ blocks, an {\bf outer jump partition} $\outer$ of size $\Delta$ for $\pi$ is a
sequence of partitions $\outer=(\nu^{(1)},\dots,\nu^{(M-1)})$ satisfying the following conditions.
\begin{enumerate}
\item $| \outer| := \sum_{i=1}^{M-1} | \nu^{(i)} | = \Delta$, and
\item For $1 \le i \le  M-1$, $\ell(\nu^{(i)}) \le c_i$ and $\nu^{(i)}_1 < c_{i+1}$.
\end{enumerate}

We say that an outer jump partition $\outer$ {\em acts} on a minimal permutation $\pi$ by left multiplication. In particular, define the action of $\outer$  as
\begin{align}
 \outer  \circ \pi =& 
 \prod_{i=1}^{{M-1}}  \left(\transp{c_{i}^*-1+\nu^{(i)}_1} \cdots \transp{c_{i}^*-\ell(\nu^{(i)})+1} \right) 
 \cdots \left( \transp{c_{i}^*-\ell(\nu^{(i)})+\nu^{(i)}_{\ell(\nu^{(i)})}} \cdots  \transp{c_{i}^*-\ell(\nu^{(i)})+1} \right) 
\cdot \pi,
\end{align}
{where the action of the product is in increasing order.}
We will call the action of each transposition above as an {\em outer jump}. 

\end{definition}

\begin{example}
\label{eg:outeronmin}
Let $\pi$ be the same minimal permutation as in Example~\ref{eg:inneronmin} {so that $c_1^* = 3$ and
$c_2^* = 7$}, and let $\outer = ((2),\emptyset)$. Then,
\[
\outer \circ \pi  =  
{
\underbrace{\transp{4} \cdot \transp{3}}_{\nu_1 = (2)}
\cdot \underbrace{\text{id}}_{\nu_2 = \emptyset}
}
 \cdot \pi = (5,2,1\, |\, 7,6,4,3\, |\, 10,9,8).
\]
{This example also clarifies the reason for naming the action of the transposition an ``outer jump". If one assigns the parts of the partition $(2)$ to the element  $3$, the size of the part is the number of ``jumps" performed by the element which moves from left from the first block and jumps over elements of the second block by wrapping around the ``outer" boundary  between the first and second blocks.}
\end{example}

\begin{definition}
A {\bf jump partition} for a minimial permutation $\pi$ with $M$ blocks of size $\Delta$ is a tuple $J=(\inner ,\outer)$ of inner and outer jump partitions such that
\[
|J| := |\inner|  + | \outer | = \Delta.
\]
\end{definition}

{The next result follows because there are a total of $2(M-1)$ partitions in a jump partition, each of which can be thought to be of a different color.}

{
\begin{proposition}
\label{prop:jump partition}
The number of jump partitions of size $\Delta$ for a minimal permutation $\pi$ with $M$ blocks is $p_{2(M-1)}(\Delta)$.
\end{proposition}
}

We {\em apply} a jump partition to a minimal permutation $\pi$ by acting on it first by $\inner$ and then by $\outer$, that is, define
\begin{equation}
J(\pi) \defeq  \outer \circ \pi \circ \inner.
\end{equation}

\begin{example}
\label{eg:jumponmin}
Let $\pi$ be the same minimal permutation as in Examples~\ref{eg:inneronmin} and \ref{eg:outeronmin},
with $\inner = (\emptyset,(1,1))$,  $\outer = ((2),\emptyset)$ and $J= (\inner ,\outer)$. 
Then,
\[
J(\pi)  =  \transp{4} \cdot \transp{3} \cdot \pi \cdot \transp{7} \cdot \transp{6}= (5,2,1\, |\, 7,6,10,4\, |\, 3,9,8).
\]
In this case, $\Delta = 4$, which is larger than the smallest column in $\lambda'$, which is 3.

\end{example}

The following is an immediate consequence of the definitions.

\begin{lemma} \label{lem:orderirrel}
The order of the action of inner and outer jump partitions is irrelevant.
\end{lemma}

For definiteness, we fix that we first apply the inner jump partition followed by the outer jump partition.

\begin{lemma} \label{lem:Jinverse}
Let $J = (\inner,\outer)$ be a jump partition for a minimal permutation $\pi$ and let $J'=(\outer,\inner)$. Then 
\[
(J(\pi))^{-1} = \inner \circ \pi \circ \outer = J'(\pi).
\]
\end{lemma}

\begin{proof}
We have
\[
(J'(\pi))^{-1} = (\inner \circ \pi \circ \outer )^{-1} = \outer \circ \pi^{-1} \circ \inner = \outer \circ \pi \circ \inner = J(\pi).
\]
Above, the second equality can be verified by using Definitions \ref{def:inner} and \ref{def:outer}, while the third equality follows since minimal permutations are involutions.
\end{proof}

\begin{example}
Suppose $\pi \in S_{53}$ is a minimal permutation corresponding to the composition $c = (14, 15, 12, 12)$ so that it is
written in one line notation as
\[
(14, \dots, 3, 2, 1\, |\,  29, 28, 27, \dots, 17, 16, 15
\, |\,  41, 40, 39, \dots, 33, 32, 31, 30\, |\,  53, \dots, 43, 42).
\]
The jump partition $J= (\inner, \outer)$ given by $\inner=((1,1),(2,1),\phi)$ and $\outer=(\phi,(3,1),(2))$ applied to $\pi$ returns
$(\transp{42}\transp{41})(\transp{28})(\transp{31}\transp{30}\transp{29})
\cdot \pi \cdot
(\transp{14})(\transp{13})(\transp{29}\transp{30})(\transp{28})$,  
and we obtain
\[
\resizebox{\hsize}{!}{
$(14, \dots, 3, 32, 2\, |\,  1, 29, 27, \dots, 17, 43, 16\, |\, 
40, 15, 39, \dots, 33, 31, 30, 28\, |\,  53, \dots, 44, 42, 41).
$
}
\]

\end{example}

We now present a complete example for $M=2$ in order to clarify the notation set up above, and to give some insight for Theorem~\ref{thm:two-col} and
Conjecture \ref{conj:more than two col}. 

\begin{example}
\label{eg:M=2}
We begin with a minimal permutation on $12$ elements with shape $\lambda = \langle 2^6 \rangle$ and enumerate the permutations that are obtained from all possible jump partitions, dividing into cases according to the number of inversions introduced.
The unique minimal permutation is given by $(6,5,4,3,2,1\, |\, 12,11,10,9,8,7)$.
{
One can compute by hand all permutations for $\Delta =1$ and $\Delta = 2$ for this $\lambda$. It turns out that there are 2 and 5 such permutations respectively.
We show below that in fact jump partitions can be used to enumerate these permutations.
}

From the minimal permutation, the permutations corresponding to $\Delta=1$ are obtained by applying jump partitions of size 1. {In each example below, we will write inner jumps $\mu$ in parentheses and
outer jumps $\nu$ in square brackets.}
The first row in each pair of permutations below is the 
unique permutation corresponding to $\Delta=0$, and number
of jumps an element is involved in is noted next to an arrow above the element. The 
second row shows the resulting permutation. The affected elements
(which differ at their positions from the original permutation) are
parenthesized.

\[
\begin{array}{c}
\mathrm{\mathbf{1 \ inner, \ 0 \ outer \ jumps} }\\

\begin{array}{cccccc|cccccc}
&&&&&(1) \\
6 & 5 & 4 & 3 & 2 & 1^{\rightarrow 1} &  12 & 11 & 10 & 9 & 8 & 7 \\
6 & 5 & 4 & 3 & 2 & (12  &  1) & 11 & 10 & 9 & 8 & 7 \\
\end{array}
\end{array}
\]

\[
\begin{array}{c}
\mathrm{\mathbf{ 0 \ inner, \ 1 \ outer \ jump}} \\

\begin{array}{cccccc|cccccc}
[1] \\
6^{\rightarrow 1} & 5 & 4 & 3 & 2 & 1   &  12 & 11 & 10 &
  9 & 8 & 7 \\ 
7) & 5 & 4 & 3 & 2 & 1 &  12 & 11 & 10 & 9 & 8 & (6 \\
\end{array}
\end{array}
\]

Proceeding in the same way for $\Delta=2$, we now list out all the
ways to divide a total of two jumps among the inner and outer
jumps. Each jump will result in one more inversion. Jumps are made so an element
does not jump over any element in its own block. 
Thus an element
can make at most as many jumps as the element ahead of it. This
restriction on the number of jumps that an element can make compared
to its successor means that the vector of jumps is a partition.

\[
\begin{array}{c}
\mathrm{\mathbf{1 \ inner, \ 1 \ outer \ jump} }\\

\begin{array}{cccccc|cccccc}
[1] &&&&&(1) \\
 6^{\rightarrow 1}  & 5 & 4 & 3 & 2 &  1^{\rightarrow 1}   &  12 & 11 &
  10 & 9 & 8 & 7 \\ 
7) & 5 & 4 & 3 & 2 & (12  &  1) & 11 & 10 & 9 & 8 & (6 \\
\end{array}
\end{array}
\]

\[
\begin{array}{c}
\mathrm{\mathbf{ 2 \ inner, \ 0 \ outer \ jumps}} \\

\begin{array}{lcccccc|cccccc}
&&&&&& (2) \\
\text{{(i)}} & 6 & 5 & 4 & 3 & 2 &  1^{\rightarrow 2}   &  12 & 11 &
  10 & 9 & 8 & 7 \\ 
& 6 & 5 & 4 & 3 & 2 & (12  &  11 & 1) & 10 & 9 & 8 & 7 \\
\end{array}\\
\vspace{0.1cm} \\
\begin{array}{lcccccc|cccccc}
& &&&&(1)& (1) \\
\text{{(ii)}} & 6 & 5 & 4 & 3 &  2^{\rightarrow 1}   &  1^{\rightarrow 1}   &  12 & 11 & 10 & 9 & 8 & 7 \\ 
& 6 & 5 & 4 & 3 & (12 & 2  &  1) &  11 & 10 & 9 & 8 & 7 \\
\end{array}\\

\end{array}
\]

\[
\begin{array}{c}
\mathrm{ \mathbf{ 0 \ inner, \ 2 \ outer \ jumps}} \\

\begin{array}{lcccccc|cccccc}
& [2] \\
\text{{(i)}} &  6^{\rightarrow 2} & 5 & 4 & 3 & 2 & 1 &  12 & 11 &
  10 & 9 & 8 & 7 \\ 
& 8) & 5& 4 & 3 & 2 & 1  &  12 & 11 & 10 & 9 & (7 & 6 \\
\end{array}\\
\vspace{0.1cm} \\

\begin{array}{lcccccc|cccccc}
& [1] & [1] \\
\text{{(ii)}} & 6^{\rightarrow 1}  &  5^{\rightarrow 1}   & 4 & 3 & 2 & 1  &  12 & 11 &
  10 & 9 & 8 & 7 \\ 
& 7 & 6) & 4 & 3 & 2 & 1  & 12 & 11 & 10 & 9 & 8 & (5\\
\end{array}
\end{array}
\]

Thus, the total number of
permutations with shape $\langle 2^6 \rangle$ with two inversions is
\[
w_{(6,6)'}^2 = p(0)p(2) +p(1)p(1)+p(2)p(0) = p_2(2) = 5,
\]
{
where we have used \eqref{pc-recur} for the second equality.
Thus, we have established by this example that $w_{(6,6)'}^2$ equals $p_2(2)$, which agrees with  Theorem~\ref{thm:two-col}.
}
\end{example}

\section{Permutations with two columns}
\label{sec:two-col}

In this section we will prove Theorem \ref{thm:two-col}. Recall that $W_\lambda$ is the set of permutations whose shape $\lambda$ is given by
$\lambda' = (s,r)$, with $s \geq r$. 
From Theorem~\ref{thm:holhweg}(1), the minimal number of inversions associated to permutations of shape $\lambda$ is $\binom r2 + \binom s2$. 

The proof of Theorem~\ref{thm:two-col} will be bijective. Since we are dealing with the case of two columns, i.e. $k=2$, we will denote a jump partition by $J = (\mu,\nu)$, where $\mu,\nu$ are partitions, for convenience. We say that a jump partition $J$ is {\em valid} for a minimal permutation $\pi$ if $\sh(J(\pi)) = \sh(\pi)$. We begin by collecting a few useful results.
The following result immediately follows from Lemma \ref{lem:Jinverse}.

\begin{corollary} \label{cor:dualvalid}
If $J = (\mu,\nu)$ is valid for $\pi$, then so is $J'=(\nu,\mu)$.
\end{corollary}

\begin{proof}
In the RS correspondence, the shape of a permutation and its inverse are the same.
\end{proof}

\begin{lemma} \label{lem:innerouterdisjoint}
If $\Delta < r$, and $J = (\mu,\nu)$ is a jump partition of size $\Delta$, the elements of the minimal permutation $\pi$ not fixed under the action of $\mu$ are disjoint from those not fixed under the action of $\nu$.
\end{lemma}

\begin{proof}
We analyze the case that the minimal permutation is $\pi_1 = (s, s-1, \dots, 1 \,  | \,  s+r , \dots, s+1)$;
the other case is similar.
From Definition~\ref{def:inner}, when the inner jump partition $\mu$ is applied to $\pi_1$, the set of elements which are not fixed is $M = \{1,\ldots,\ell(\mu),s+r-\mu_1+1,\ldots,s+r\}$. Similarly from Definition~\ref{def:outer}, when the outer jump partition $\nu$ is applied to $\pi_1$, the set of elements that are not fixed is $N = \{s - \ell(\nu) +1, s - \ell(\nu),\cdots,s+\nu_1\}$.
Note that we have 
\begin{align*}
\ell(\mu) +\ell(\nu) \le \sum_{i=1}^{\ell(\mu)} \mu_i + \sum_{i=1}^{\ell(\nu)} \nu_i = \Delta < r \le s.
\end{align*}
From this, we conclude that $\ell(\mu) +\ell(\nu) < s+1$
and $\mu_1 + \nu_1 < r+1$.
These inequalities imply that $s-\ell(\nu) +1 > \ell(\mu)$
and $s+r-\mu_1+1 > s + \nu_1$,
and hence that $ M \cap N = \emptyset$.
\end{proof}

\begin{proposition}
\label{prop:jumpbicolor}
The number of jump partitions of size $\Delta$ is equal to $p_2(\Delta)$.
\end{proposition}

\begin{proof}
This follows immediately from the definition of a jump partition.
\end{proof}

We divide the ideas of the proof into the following two subsections. 
In Section~\ref{sec:lowerbd}, we show that starting from a minimal permutation $\pi \in W_\lambda$ and applying a valid jump partition $J$ of size $\Delta <r$ to it, results in a unique permutation in $W_\lambda^\Delta$. In Section~\ref{sec:upperbd}, we show that when $\Delta <r$, every permutation in $W_\lambda^\Delta$ can be obtained uniquely from a minimal permutation and a valid jump partition of size $\Delta$. 
Finally, we stitch together the ideas to complete the proof in Section~\ref{sec:complete-proof}.

\subsection{Lower bound}
\label{sec:lowerbd}

\begin{theorem} 
\label{thm:jumpseqgivesperm}
Let $\Delta <r$.  Every valid jump partition $J$ of size $\Delta$ for a minimal permutation of shape $\lambda$, where $\lambda' = (s,r)$, gives a permutation in $W_\lambda^\Delta$.
\end{theorem}

We will prove Theorem~\ref{thm:jumpseqgivesperm} by showing that the action of every valid jump partition on a minimal permutation can also be obtained by Knuth moves while keeping track of the number of inversions. Before proving the result, we give an example to illustrate why this is nontrivial.

\begin{example}
Consider the minimal permutation $\pi = (5,4,3,2,1\, |\, 11,10,9,8,7,6)$ corresponding to shape $\lambda$ with $\lambda' = (6,5)$, and the inner jump partition which is the single partition $\mu = (2,2)$. This corresponds to $\pi \transp{5} \transp{6} \transp{4} \transp{5}$. The first three actions of transpositions on $\pi$ are $K_+$ moves. For instance, the permutation $\pi \transp{5}$ corresponds to interchanging 1 and 11, which is a  $K_+$ move because 2 is to the immediate left of 1. Similarly, after the action of the other two transpositions, we get the permutation $\pi' = (5,4,3,11,2\, |\, 10,1,9,8,7,6)$. But 
$\pi' \transp{5}$ corresponds to interchanging 2 and 10, which is no longer a  $K_+$ move.

{To solve this problem in this case, we first consider $\pi' \transp{6}$, which interchanges $1$ and $9$ and is a $K_+$ move. Now, the interchange of 2 and 10 is indeed a $K_+$ move. Finally, we interchange $9$ and $1$ back as a $K_-$ move. The strategy of proof of 
Theorem~\ref{thm:jumpseqgivesperm} generalizes this idea.}
\end{example}

\begin{proof}[Proof of Theorem~\ref{thm:jumpseqgivesperm}]

There are two possibilities for the minimal permutation $\pi$ when $r<s$, either
$(r,\dots,1\, |\, n,n-1,\dots,r+1)$ or $(s,\dots,1\, |\, n,n-1,\dots,s+1)$,
which are identical when $r=s$. We will only prove the result for the first case as the other can be argued along similar lines.

Consider the case when the jump partition is purely inner, i.e. $J = (\mu, \emptyset)$. Write the partition $\mu$ of $\Delta$ as $\mu = (\mu_1, \dots, \mu_k)$ where $r > \mu_1 \geq \cdots \geq \mu_k > 0$ and $k<r$.
The algorithm is as follows. The idea is to enlarge the inner jump partition to a strict partition (i.e. a partition with distinct parts) $\widetilde{\mu}$ containing $\mu$ whose action on $\pi$ can be decomposed into a sequence of  $K_+$ moves. Next, by applying a specific sequence of  $K_{-}$ moves to the resulting permutation $\pi \circ \tilde \mu$, we obtain $\pi \circ \mu$. The claim will then follow since the application of Knuth moves to a permutations preserves its shape.
Define 
\(
\widetilde{\mu} \defeq (\mu_1+k-1, \dots, \mu_{k-1}+1,\mu_k).
\)
We will first compute $\pi \circ \widetilde{\mu}$.

Recall that this means, according to Definition~\ref{def:inner}, that we first move 1 $\mu_1+k-1$ steps to the right, followed by moving 2 $\mu_2+k-2$ steps to the right, and so on, up to moving $k$ right $\mu_k$ steps. We claim that each of these is a  $K_+$ move. 
Before we do so, let us first verify that $\widetilde \mu$ is valid for $\pi$. First, notice that since $\Delta < r$, the numbers 1 through $\Delta$ do appear in decreasing order in $\pi$ as in the example above. 
Next, notice that the number 1 moves the largest number of steps to the right and since
$\widetilde \mu$ is a strict partition, the value $i$ always moves less than $i-1$. 
Since 1 is at position at most $s$ in $\pi$, its final position is at most $s+\mu_1+k-1$ in $\pi \circ \widetilde \mu$. We want to make sure this is less than $n$. This follows since
\[
\mu_1+k-1 \leq \mu_1+\cdots+\mu_k = \Delta < r
\]
showing that $\pi \circ \widetilde \mu$ is well-defined.

Now, each elementary transposition of 
the inner jump partition $\widetilde \mu$ is a  $K_+$ move because it is always
of the form
\[
\cdots ,i,p,p-1, \cdots \rightarrow \cdots ,p,i,p-1, \cdots
\]
with $p>i$ because $\widetilde \mu$ is a strict partition. It is easy to check by the construction that $\pi \circ \widetilde \mu$ has $\Delta+ \binom{k}{2}$ more inversions than $\pi$. 

We will now perform $\binom{k}{2}$ $K_-$ moves starting from the permutation $\pi \circ \widetilde{\mu}$ to obtain $\pi \circ \mu$. We focus on the relative positions of $1,\dots,k-1,k$ in $\pi \circ \widetilde{\mu}$. They appear in the form
\[
\resizebox{\hsize}{!}{
$
\cdots, x, k, x-1,\dots,x-z_1,k-1,x-z_1-1,\dots,x-z_2,k-2,
\dots, x-z_{k-1},1,x-z_{k-1}-1,\cdots 
$
}
\]
with a gap of at least one between the numbers $1,\dots,k$ and 
all other elements in consecutive decreasing order. 
The actual values of 
$z_1,\dots,z_{k-1}$ depend on a simple way on $\widetilde \mu$, but are not important for the argument. 
Note that $x-z_{k-1}-1$ is larger in value than $k$. 
Starting from $\pi \circ \widetilde \mu$, we first move $k-1$ to the left by performing an elementary transposition. These are $K_-$ moves of the form
\[
\cdots ,p,i,p-1, \cdots \rightarrow \cdots ,i,p,p-1, \cdots
\]
with $p>i$. We continue similarly moving all the elements $k-2,\dots,1$ one by one in that order until they reach their respective positions in $\pi \circ \mu$. Each of these moves is a  $K_-$ move of the same form as above because of the gap between $i$ and $i+1$ for $i \in \{1,\dots,k-1\}$. Hence, $\pi \circ \mu 
\in W_\lambda^\Delta$.

In the case that $J = (\emptyset,\nu)$ of size $\Delta$, by Corollary \ref{cor:dualvalid}, $\nu \circ \pi \in W_\lambda^\Delta$. Note that in this case the corresponding transpositions are $KD_{+}$ and $KD_{-}$ moves.

When $J = (\mu,\nu)$ with $\mu = (\mu_1,\dots,\mu_k)$ and
$\nu = (\nu_1,\dots,\nu_j)$, we can construct  strict partitions $\widetilde \mu
 = (\mu_1+k-1, \dots, \mu_{k-1}+1,\mu_k)$ and
$\widetilde \nu  = (\nu_1+j-1, \dots, \nu_{j-1}+1,\nu_j)$ as above.

From the argument above, it is clear that one can separately apply either $\tilde \mu$ or $\tilde \nu$ to $\pi$. The only issue that could arise is that there could be interference between these two actions in the sense that both could act on the same entries. We will now show that such a situation cannot arise.
Following the notation of Example~\ref{eg:M=2}, we depict the jump partition on one of the minimal permutations listed above as
\[
\begin{array}{ccccccc|ccc}
[\nu_1+j-1] & \dots & [\nu_j] & & (\mu_k) & \dots\dots &(\mu_1+k-1) \\
r &  \dots & r-j+1 & \dots & k & \dots \dots & 1 &  n &  \dots  & r+1.
\end{array}
\]
Since $|\mu| + |\nu| = \Delta < r$, $k + j< r$ and no value among $1,\cdots r$ can be acted upon by both $\tilde \mu$ and $\tilde \nu$.
Among the values $r+1,\dots,n$, the values that are moved by $\tilde \nu$ are $n-\nu_1-j+2,\dots,n$ and those moved by $\tilde \mu$ are $r+1,\dots,r+\mu_1+k-1$. But since $\mu_1+k-1 \leq |\mu|$ and $\nu_1+j-1 \leq |\nu|$, no value is affected by both $\tilde \mu$ and $\tilde \nu$. Since $\tilde \mu$ and $\tilde \nu$ contain $\mu$ and $\nu$ in their Young diagrams, similar arguments show that $ \nu$ and $\mu$ act disjointly. This completes the proof of the claim.
\end{proof}

\begin{theorem}
\label{thm:Jpigivesuniqueperm}
Let $\lambda \vdash n$ be a partition such that $\lambda' = (s,r)$. Then, for each $\pi \in W_\lambda^0$ and jump partition $J$ with size $|J| = \Delta< r$ that is valid for $\pi$, the permutation $J(\pi)$ arises in a unique way. That is,
\begin{enumerate}
\item If $J_1 \ne  J_2$, then $J_1(\pi) \ne J_2(\pi)$. \label{part:sameminimal}
\item If $\pi_1 \ne \pi_2$ and $J_1$ (resp. $J_2$) is a valid jump partitions for $\pi_1$ (resp. $\pi_2$), then $J_1(\pi_1) \ne J_2(\pi_2)$. \label{part:diffminimal}
\end{enumerate}
\end{theorem}

\begin{proof}
Let $J_1=(\mu_1,\nu_1)$ and $J_2 = (\mu_2,\nu_2)$. We begin by showing part \eqref{part:sameminimal}. 
Let $I_1$ and $O_1$ (resp. $I_2$ and $O_2$) be the set of elements moved in $\pi$ under the action of $\mu_1$ and $\nu_1$ (resp. $\mu_2$ and $\nu_2$) respectively. Since $\Delta <r$, by Lemma \ref{lem:innerouterdisjoint}, $I_1$ and $O_1$ are disjoint, and so are $I_2$ and $O_2$. 
Since $J_1 \ne J_2$, either $\mu_1 \ne \mu_2$ or $\nu_1 \ne \nu_2$. Suppose $\mu_1 \ne \mu_2$ and let $k$ be the largest part in which they differ, i.e. $\mu_{1,k} \ne \mu_{2,k}$. Then, using Lemma \ref{lem:innerouterdisjoint}, the element $k \in I_1 \cap I_2$ is not in either of $O_1$ and $O_2$ and appears in different positions in $J_1(\pi)$ and $J_2(\pi)$.

For part \eqref{part:diffminimal}, in the case that $s=r$, there is exactly one minimal permutation, so there is nothing to prove. Therefore, we assume that $s>r$. Let $\pi_1$ and $\pi_2$ be two minimal permutations:
\begin{equation}
\begin{split}
\label{minperms-twocol}
\pi_1 & = (s, s-1, \dots, 1 \,  | \,  s+r , \dots, s+1)  = ( Y , X \,  | \,  Z) \\
\pi_2 & = (r, r-1, \dots, 1 \,  | \,  s+r , \dots, r+1)=  (X \, | \,  Z, Y)
\end{split}
\end{equation}
where we use $X, Y$ and $Z$ to denote the blocks $(r, r-1, \dots, 1)$,  $(s,  s-1, \dots, r+1)$ and  $(s+r,  \dots,  s+1)$ respectively.

We will show that applying a jump partition to these minimal permutations results in constraints on the structure of the resulting permutations. In particular, we will show that all the elements that are fixed  in $J_2(\pi_2)$ cannot appear at the same positions as in $J_1(\pi_1)$. 

First, we claim that for some contiguous set of indices $I$ such that $1 \le i \le r$ for $i \in I$, $(J_2(\pi_2))_i = r-i+1$ for $i \in I$. That is, when $J_2$ is applied to $\pi_2$, it fixes $i$ for each $i \in I$.
We verify that since $\Delta <r$, a contiguous block of $X$ consisting of $r-\ell(\mu_2)-\ell(\nu_2) > 1$ elements will remain fixed under the action of $J_2$. To see this, note that $\mu_2$ fixes the first $r-\ell(\mu_2)$ elements of $X$ while $\nu_2$ fixes the last $r-\ell(\nu_2)$ elements of $X$. Thus $r-\ell(\mu_2)-\ell(\nu_2)$ contiguous elements of $X$ are fixed by both $\mu_2$ and $\nu_2$ and hence by $J_2$.

Suppose now that $i \in X$ is fixed by $J_2$ at position $r-i+1$ in $\pi_2$. In the permutation $\pi_1$, the element $i$ is at position $s-i+1 > r-i+1$. When $J_1$ is applied to $\pi_1$, $i$ may either remain fixed, or its position may be changed due to an inner or outer jump. If it stays fixed, then its position in $J_1(\pi_1)$ is different from its position in $J_2(\pi_2)$. If the position of $i$ is changed due to either an inner or an outer jump, then its position in $J_1(\pi_1)$ only increases and therefore cannot be equal to $r-i+1$. 
\end{proof}
 
We now prove Theorem~\ref{thm:more than two col} using the ideas of Theorems~\ref{thm:jumpseqgivesperm} and
\ref{thm:Jpigivesuniqueperm}.
Recall that $\pi$ is a minimal permutation for $\lambda$, with $\lambda' = \langle t_1^{m_1}, \dots, t_k^{m_k} \rangle$, and $J = (\inner, \outer)$ is a jump partition of size $\Delta < t_k$.

\begin{proof}[Proof of Theorem~\ref{thm:more than two col}]
When $\inner$ or $\outer$ consists of a single partition, the argument is the same as that of Theorem~\ref{thm:jumpseqgivesperm}.
When $\outer$ is empty, the same argument works again because
the partitions $\mu_i,\mu_{i+1}$ in the inner jump partition $\inner$ (for example) cannot interact (in the sense that 
the elements moved by them are disjoint) because $\Delta$ is small. This is also the case when $\inner$ is empty. Suppose now that $\inner = (\mu_1,\dots,\mu_k)$ and $\outer = (\nu_1,\dots,\nu_k)$ are both nonempty.

The only cases where elements can be moved simultaneously by both the inner and the outer jump partition is
when both $\mu_i$ and $\nu_{i+1}$ are nonempty, and when both $\nu_i$ and $\mu_{i+1}$ are nonempty.
We will now show that in both cases, the moves made are Knuth or dual Knuth transformations. It will suffice to consider 
the first case, i.e., when both $\mu_i$ and $\nu_{i+1}$ are nonempty. The argument in the other case is similar.

Consider the $i$'th, $(i+1)$'th, and the $(i+2)$'th block in $\pi$, labelled as follows.
\[
\cdots \: |\: a_1, \dots, a_r \: | \: b_1, \dots, b_s \: | \: c_1, \dots, c_t \: | \cdots,
\]
where the $a_j$'s, the $b_j$'s and the $c_j$'s are consecutively decreasing, 
$b_s = a_1+1$ and $c_t = b_1+1$.
The action of $\mu_i$ interchanges the rightmost elements of the $a$-block with the leftmost ones of the $b$-block. 
We first perform the $\mu_i$ action, which is valid by Theorem~\ref{thm:jumpseqgivesperm}. 
At this point, some of the $b_j$'s are in the $a$-block and some of the $a_j$'s are in the $b$-block. If $\mu_{i+1}$ is nonempty, some of the $b_j$'s are in the $c$-block. Schematically, the action of $\mu_i$ and $\mu_{i+1}$ cause the blocks to look as follows.
\[
\cdots \: |\: a_1, \dots, a_{r-1}, b_1 \: | \: a_r, b_2, \dots, b_{s-1}, c_1 \: | \: b_s, c_2, \dots, c_t \: | \cdots,
\]
The key observation is that the relative position of the elements of the $b$- and $c$-blocks are unchanged among themselves at this point. We now perform the outer jumps according to $\nu_{i+1}$, which interchanges the leftmost elements of the $b$-block with the rightmost ones of the $c$-block. Since the $KD_+$ and $KD_-$ moves do not depend on the locations of the values being interchanged, the action of $\nu_{i+1}$ is valid after the action of $\mu_i$ if and only if it is valid before the action of $\mu_i$. We therefore perform the action of $\nu_{i+1}$ exactly as prescribed in Theorem~\ref{thm:jumpseqgivesperm}.
As always, the fact that $\Delta < t_k$ guarantees that these interchanges will be dual Knuth moves.

Now, we have to show that if $J_1$ and $J_2$ are valid distinct jump partitions for minimal permutations $\pi_1$ and $\pi_2$ respectively, then $J_1(\pi_1) \neq J_2(\pi_2)$. Since $J_1 \neq J_2$, there exists an index $i$ such that either $\mu_{1,i} \neq \mu_{2,i}$ or $\nu_{1,i} \neq \nu_{2,i}$. Suppose that the former is true. If $\pi_1 = \pi_2$, we repeat the idea of the proof of Theorem~\ref{thm:Jpigivesuniqueperm}(i). 

Now, suppose $\pi_1 \neq \pi_2$. By Theorem~\ref{thm:holhweg}(3), $c_1 \neq c_2$ are two distinct permutations of $\lambda'$ such that $\pi_1$ (resp. $\pi_2$) are minimal permutations corresponding to $c_1$ (resp. $c_2$). Let $i$ be the smallest integer such that $c_{1,i} \neq c_{2,i}$. Then focus on the $i$'th and $(i+1)$'th blocks in $\pi_1$ and $\pi_2$. The structure of these blocks is essentially the same as that in \eqref{minperms-twocol}, the sole difference being that $Z$ in both blocks is different. But that does not affect the argument, since one is only looking at fixed points within $X$ in both permutations.

We have thus shown that each jump partition gives rise to a unique permutation in $W_\lambda^\Delta$.
{
Since there are $\binom M {m_1,\dots,m_k}$ minimal permutations by Theorem~\ref{thm:holhweg}(2) and $p_{2(M-1)} (\Delta)$ jump partitions acting on each of them by Proposition~\ref{prop:jump partition} giving rise to distinct permutations, we have the required lower bound for the cardinality of $W_\lambda^\Delta$, completing the proof.}
\end{proof}

\subsection{Upper bound}
\label{sec:upperbd}

In this section, we will prove Theorem~\ref{thm:permgivesuniquejpimin}, i.e. we will show that every permutation in $W_\lambda^\Delta$ can be obtained from a unique pair of a valid jump partition and a minimal permutation of shape $\lambda$ (by applying the jump partition to the permutation).

\begin{theorem}
\label{thm:permgivesuniquejpimin}
Let $\lambda \vdash n$ be a partition such that $\lambda' = (s,r)$ . Let $\Delta <r$ and $\sigma \in W_\lambda^\Delta$. Then, there is a unique tuple $(\pi,J)$ where $\pi \in W_\lambda^0$, $J$ is a valid jump partition of size $|J| = \Delta$ and $\sigma = J(\pi)$.
\end{theorem}

We will prove Theorem~\ref{thm:permgivesuniquejpimin} by considering the graphical representation of permutations, i.e, given $\pi \in S_n$, the set of points in $[1,n]^2$ given by coordinates $(i,\pi_i)$ for $1 \leq i \leq n$.
We know that there exists a longest decreasing subsequence (LDS) of length $s$ in $\pi$. Call the corresponding set of points $S$, and the remaining, $R$. We focus on the smallest rectangle enclosing $S$. The rectangle divides $[1,n]^2$ into various regions, as depicted in Figure~\ref{fig:LDS-2column}.

We will first need a result for the longest increasing subsequence (LIS) for two-rowed permutations.
Let $\pi$ be a permutation whose shape is $(s,r)'$. If a partial permutation $\pi'$ is formed by removing an element $p$ from $\pi$ in its one-line notation, we denote it by $\pi' = \pi \setminus \{p\}$. The notion of LIS also makes sense for $\pi'$ and $\pi'$ can be standardized to a permutation of size $[n-1]$ in the obvious way. We need the following result about the LIS.

\begin{proposition}
\label{prop:LIS-remove-pt}
Let $\lambda = (s,r)$ and $\pi$ be a permutation with shape $\lambda$. If $s>r$, there exists $p \in [n]$ such that $\pi \setminus \{p\}$ has an LIS of length $s-1$. Moreover, $p$ is in every LIS of $\pi$.
\end{proposition}

\begin{proof}
It is clear that $\pi$ has an LIS of length $s$. Let $p_1,\dots,p_s$ be an LIS of $\pi$. We use the following standard fact about any LIS $q_1,\dots,q_t$ of a permutation $\sigma$:
$q_i$ is inserted in the $i$'th column in the first row for all $i$ in the Robinson-Schensted algorithm. 
Now, since the second row of $\pi$ has length $r<s$, there are exactly $r$ bumpings, and hence $s-r$ positions in the first row which are not bumped. All the elements in those positions are a part of every LIS of $\pi$. Let $p$ be one of those elements.
If $p$ is removed, then it is clear that the resulting partial permutation has an LIS of length $s-1$. Clearly, every LIS must therefore contain $p$ since otherwise, the length of an LIS for $\pi \setminus \{p\}$ would be $s$.
\end{proof}

By considering the reverse permutation $\pi^R(i) = {(\pi_n,\dots,\pi_1)}$ in Proposition~\ref{prop:LIS-remove-pt}, we immediately obtain the following corollary.

\begin{corollary}
\label{cor:LDS-remove-pt}
Let $\pi \in S_n$ with shape $(s,r)'$ with $s>r$. Then there exists $p \in [n]$ such that $\pi \setminus \{p\}$ has an LDS of length $s-1$. Moreover $p$ is in every LDS of $\pi$.
\end{corollary}

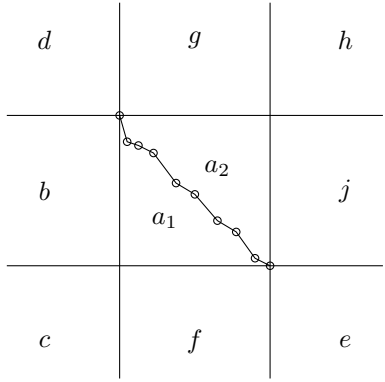
\begin{figure}[htbp!] 
\begin{center}
\begin{tikzpicture} [>=triangle 45, x=0.5cm,y=0.5cm]
\draw (1,1) node {$c$};
\draw (1,5) node {$b$};
\draw (1,9) node {$d$};
\draw (5,1) node {$f$};
\draw (5.6,5.6) node {$a_2$};
\draw (4.2,4.2) node {$a_1$};
\draw (5,9) node {$g$};
\draw (9,1) node {$e$};
\draw (9,5) node {$j$};
\draw (9,9) node {$h$};
\draw [-] (3,0) -- (3,10);
\draw [-] (7,0) -- (7,10);
\draw [-] (0,3) -- (10,3);
\draw [-] (0,7) -- (10,7);
\draw [-] (3,7) node[shape=circle,inner sep = 1pt,draw] {}
				-- (3.2,6.3) node[shape=circle,inner sep = 1pt,draw] {}
				-- (3.5,6.2) node[shape=circle,inner sep = 1pt,draw] {} 
				-- (3.9,6.0) node[shape=circle,inner sep = 1pt,draw] {} 
				-- (4.5,5.2) node[shape=circle,inner sep = 1pt,draw] {} 
				-- (5.0,4.9) node[shape=circle,inner sep = 1pt,draw] {} 
				-- (5.6, 4.2) node[shape=circle,inner sep = 1pt,draw] {} 
				-- (6.1, 3.9) node[shape=circle,inner sep = 1pt,draw] {} 
				-- (6.6, 3.2) node[shape=circle,inner sep = 1pt,draw] {} 
				--  (7,3) node[shape=circle,inner sep = 1pt,draw] {};
\end{tikzpicture}
\caption{The longest decreasing subsequence $S$ enclosed in the rectangle formed by the topmost and bottommost point in $S$. 
The remaining points of the permutation are in one of the marked regions.
\label{fig:LDS-2column}}
\end{center}
\end{figure}

\begin{lemma} \label{lem:R-points-on-one-side}
Let $\pi \in W_{(s,r)'}^\Delta$ with $\Delta < r$.
Then there exists a longest decreasing subsequence of $\pi$ with points $S$ enclosed by a rectangle 
which demarcates $[1,n]^2$ into regions as shown in Figure~\ref{fig:LDS-2column} 
such that the points in $R$ must either all appear in regions $\{b,c,f\}$ or all in $\{g,h,j\}$ and form a decreasing subsequence.
\end{lemma}

\begin{proof}
The proof will be divided into five claims.
The first three of these claims are true for any longest decreasing subsequence $S$ and we prove these assuming $S$ is arbitrary. To prove the fourth {and the fifth}, we require that the subsequence satisfy an extra property, which we specify below. 

First of all, note that for any $S$ none of points in $R$ can be in regions labeled $d$ and $e$ since that would imply the existence of a longer LDS together with the points in $S$.  Let $R_1$ be the set of points of $R$ in $a_1 \cup b \cup c \cup f$ and let $R_2$ be the points in $a_2 \cup g \cup h \cup j$. 
We will consider the case that there exists a point in $a_1 \cup b \cup c \cup f$ and show that $R_2$ and $a_1$ are empty.
The reverse case is similar.
Throughout the proof, we will use the phrase ``increase/decrease of points'' to mean the increase/decrease
of the ordinates of points as we move from left to right.

An outline of the steps we take in the proof is as follows. We will show the following three claims for any longest decreasing subsequence $S$.
\begin{enumerate}[label=(\roman*)]
\item Points in $R_1$ and $R_2$ must decrease from left to right.

\item If there is a point of $R_1$ in $c$, then $R_2 = \emptyset$. 

\item If there are points of $R_1$ in both $b$ and $f$, then $R_2 = \emptyset$.
\end{enumerate}

Among the LDSs of a permutation $\pi$, we refer to those whose highest point has smallest possible $x$-coordinate as a \textit{leftmost} LDS or LLDS. Note that the LLDS need not be unique.
\begin{enumerate}[label=(\roman*),resume]
\item For any LLDS $S$ of $\pi$, if the points of $R_1$ appear in $b \cup a_1$ or in $a_1 \cup f$, then $R_2 = \emptyset$.

\item[(v)] There are no  points in $a_1$.
\end{enumerate} 

We list certain conditions that we will use repeatedly in the proof.
\begin{itemize}
\item {\bf LDS bound:} The length of the longest decreasing subsequence is $s$.
\item {\bf LIS bound:} The length of the longest increasing subsequence is $2$.
\item {\bf Inversion bound:} {By the assumption of the lemma, $\Delta < r$. Therefore,} the number of inversions is strictly less than $\binom s2 + \binom r2 + r$. Since there are $\binom s2$ and $\binom r2$ inversions among the points in $S$ and $R$ respectively, the number of inversions between the points of $S$ and $R$ must be less than $r$.
\end{itemize}

Let $S$ be an arbitrary longest decreasing subsequence of $\pi$.
\begin{enumerate}[wide, labelwidth=!, labelindent=0pt,label=(\roman*)]
\item \label{item:noincrease}
Consider two points in $R_1$ that are increasing. First, note that these points cannot be in $b \cup c \cup f$ because together with a third point of $S$ (either the topmost or bottom-most), they violate the LIS bound. So, at least one of the points is in $a_1$. The higher of the two points is thus necessarily in $a_1$. Call this point $p_1$. 
There cannot be a point of $S$ to the northeast of $p_1$ because that violates the LIS bound. But now, by including $p_1$ in $S$, we violate the LDS bound by including $p_1$ in the LDS; see Figure~\ref{fig:LIS-LDS-bounds} for an illustration. Therefore, the points of $R_1$ must be in decreasing order. 
Similarly, the points of $R_2$ must also 
decrease from left to right.

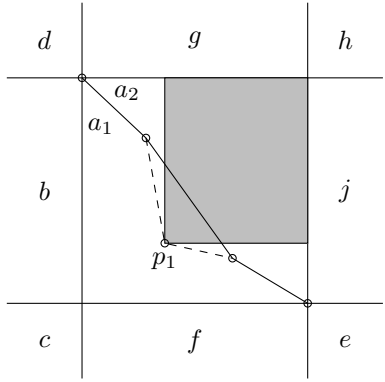
\begin{figure}[htbp!] 
\begin{center}
\begin{tikzpicture} [>=triangle 45, x=0.5cm,y=0.5cm]
\draw (1,1) node {$c$};
\draw (1,5) node {$b$};
\draw (1,9) node {$d$};
\draw (5,1) node {$f$};
\draw (3.2,7.6) node {$a_2$};
\draw (2.5,6.7) node {$a_1$};
\draw (5,9) node {$g$};
\draw (9,1) node {$e$};
\draw (9,5) node {$j$};
\draw (9,9) node {$h$};
\draw [-] (2,0) -- (2,10);
\draw [-] (8,0) -- (8,10);
\draw [-] (0,2) -- (10,2);
\draw [-] (0,8) -- (10,8);
\node [below] at (4.2, 3.6) {$p_1$};
\draw [dashed] (3.7, 6.4) -- (4.2, 3.6) -- (6.0, 3.2); 
\draw [fill = lightgray] (4.2, 3.6) rectangle (8,8);
\draw [-] (2,8) node[shape=circle,inner sep = 1pt,draw] {}
				-- (3.7, 6.4) node[shape=circle,inner sep = 1pt,draw] {} 
				-- (6.0, 3.2) node[shape=circle,inner sep = 1pt,draw] {} 
				--  (8,2) node[shape=circle,inner sep = 1pt,draw] {};
\draw (4.2,3.6) node[shape=circle,inner sep = 1pt,draw] {};
				
\end{tikzpicture}
\caption{Illustration of the violation of the LDS bound used in part~\ref{item:noincrease}. No points of $S$ lie in the shaded region to the northwest of $p_1$. Then $p_1$ can be added to $S$ to form a larger LDS. Similar ideas are used in parts~\ref{item:regionc} and \ref{item:bandf}.
\label{fig:LIS-LDS-bounds}}
\end{center}
\end{figure}

\item \label{item:regionc}
Suppose that there is a point of $R_1$ in region $c$. Then, none of the other points of $R_2$ can be in $g \cup h \cup j$ by the LIS bound. If there is a point of $R_2$ in $a_2$ ($p_2$, say), then there can be no points of $S$ to the southwest of $p_2$ by the LIS bound, analogous to the argument used in part~\ref{item:noincrease} as shown in Figure~\ref{fig:LIS-LDS-bounds}. But now, by including $p_2$ in $S$, we violate the LDS bound.

\item \label{item:bandf}
Suppose there is a point of $R_1$ each in $b$ and $f$. Then there cannot be points of $R_2$ in $\{g,h,j\}$ by the LIS bound. Thus, all points of $R_2$ must be in $a_2$. 

If any point in $a_2$ is to the northeast of any point in $b$ or in $f$, then by an argument similar to that in part~\ref{item:regionc}, either the LIS or the LDS bound is violated.
Therefore, the points of $b,a_2,f$ form a decreasing subsequence of length $r$. There are at least $r$ inversions between the extreme points of $S$ and the points of $R$, which violates the inversion bound.

\item \label{item:a1andb}
In this part of the proof, we will argue by contradiction. 
If $\pi$ is a counter example to the claim of (iv), then there is an LLDS $S$ of $\pi$ such that the points of $R_1$ appear in $b \cup a_1$ or $a_1 \cup f$ and $R_2 \ne \emptyset$.

First, suppose $s>r$. Let $\pi$ be a minimal such counterexample meaning that every subpermutation $\pi'$ of $\pi$ which satisfies the inversion bound satisfies (iv) (with the sets $a_1',a_2',b', \ldots$ defined analogously).
Let $\pi' = \pi \setminus \{p\}$ where $p$ is the point guaranteed to exist by Corollary \ref{cor:LDS-remove-pt}. By the corollary, the LDS of $\pi'$ has length $s-1$. Let $S' = S \setminus \{p\}$ so that $S'$ is an LDS of $\pi'$.
We will show that $\pi$ cannot be a minimal counterexample. Let $R_1'$ and $R_2'$ be the points corresponding to $S'$. Since the points of $R_1$ appear in $a_1 \cup b$ (resp. $a_1 \cup f$), it can be verified that the points of $R_1'$ appear in $a_1' \cup b'$ (resp. $a_1' \cup f'$). Also, since $R_2' = R_2$, we have that $R_2' \ne \emptyset$. Finally, since $\pi$ satisfies the inversion bound, we have
\begin{align}
\Delta = \inv(\pi) - \binom{s}{2} - \binom{r}{2} < r.
\end{align}
By deleting $p$ from $\pi$, we remove at least $s-1$ inversions (those involving $p$ and the other elements of $S$). Therefore,
\begin{align}
\inv(\pi') \le \inv(\pi) - (s-1).
\end{align}
Also, we have
\begin{align}
\Delta' = \inv(\pi') - \binom{s-1}{2} - \binom{r}{2}.
\end{align}
Combining these three bounds, we get
\begin{align*}
\Delta' \le \inv(\pi) - (s-1) - \binom{s-1}{2} - \binom{r}{2} = \Delta < r.
\end{align*}
That is, $\pi'$ satisfies the inversion bound, contradicting the minimality of $\pi$.

Next, we consider the case that $s=r$. The points of $R_1$ may lie in $a_1 \cup b$ or in $a_1 \cup f$. Applying the arguments
from parts~\ref{item:noincrease}, \ref{item:regionc} and \ref{item:bandf} to $R_2$, the points of $R_2$ can only lie in $g \cup a_2$ or $a_2 \cup j$. It suffices to suppose that the points of $R_2$ appear in $a_2 \cup j$ and the arguments in the other case are analogous. 

Let us suppose that $\pi$ has shape $(s,s)'$. 
Fix an LLDS of $\pi$.
Let $\pi'$ be the subpermutation of $\pi$ obtained by deleting the topmost point of $S$, which we call $v_1$. 
First of all, note that since the LDS of $\pi$ is $s$, the LIS of $\pi'$ must be two (unless $s=1$ in which case the lemma is trivial). Thus $\lambda'$, the shape of $\pi'$, also has two columns.
It is impossible that the LDS of $\pi'$ be of length $s-1$ since $\pi'$ has a total of $2s-1$ elements and its partition has 2 columns.  Thus the LDS of $\pi'$ must be of length $s$.

\noindent \textit{Case 1: $R_1 \subset a_1 \cup f$ and $R_2 \subset a_2 \cup j$:}

In this case, since all the elements of $\pi'$ are to the southeast of $v_1$, in particular, the elements of the LDS in $\pi'$ are to the southeast and so $v_1$ could be added to them to obtain a longer decreasing subsequence in $\pi$.

\noindent \textit{Case 2: $R_1 \subset a_1 \cup b$ and $R_2 \subset a_2 \cup j$:}

An LDS of $\pi'$ of length $s$ cannot have all of its points to the southeast of $v_1$, since otherwise we would get a longer LDS in $\pi$. Therefore, it must use points in $b$. However since it is also an LDS of $\pi$ this contradicts our assumption that we chose a LLDS in $\pi$.

\item \label{item:noa1}
If there is a point in $a_1$, then there cannot be a point in $c$ by part \ref{item:noincrease}. 
Moreover, there cannot be points in  $b \cup f$ because we would violate the inversion bound with the
extreme points of $S$. This argument also shows that there has to be at least one point in $c$.

\end{enumerate}

This completes the proof of the lemma.
\end{proof}

We are now in a position to prove the main result of this section.

\begin{proof}[Proof of Theorem~\ref{thm:permgivesuniquejpimin}]

Consider the permutation diagram of $\sigma$. Let $S$ be an LDS of $\sigma$ as is guaranteed by Lemma \ref{lem:R-points-on-one-side}.
Without loss of generality, assume that all points of $R$ appear in regions $b,c,f$ (see Figure~\ref{fig:LDS-2column}) with $p,q,r-p-q$ points respectively. 
We will first build the outer jump partition $\nu$.
Since the points of $R$ are in decreasing order, note that the coordinates of the points in $b$ and $c$ are $(1,n_1), \dots, (p,n_{p}),(p+1,n_{q}-1),\dots,(p+q,n_{q}-q)$, where  $n_1 > \cdots > n_{p} > n_q$. 
Consider the point $(p,n_p)$ which has the lowest ordinate in region $b$.
Suppose that there are $k_p$ points in $S$ which have lower ordinates than this one.
First of all, note that these $k_p$ points must necessarily be consecutive, and hence  have coordinates $(n-k_p+1,n_p-1),\dots,(n,n_p-k_p)$.

We now perform the transformation replacing the points $(p,n_p),(n-k_p+1,n_p-1)$ by
$(p,n_p-1),(n-k_p+1,n_p)$. It is easy to see that this new permutation also has shape $\lambda'$ and has one less inversion than $\sigma$. A little thought shows that the transformation from the latter to the former is encoded by a single outer jump.
We then interchange the pair 
$(p,n_p-1),(n-k_p+2,n_p-2)$ by $(p,n_p-2),(n-k_p+2,n_p-1)$, and continue this way until we replace $(p,n_p-k_p+1),(n,n_p-k_p)$ by $(p,n_p-k_p),(n,n_p-k_p+1)$. At this point, the lower boundary of the region $a$ shifts up by 1 and the point $(p,n_p-k_p)$ is now in region $c$. We have thus performed $k_p$ outer jumps to reduce the number of points in $b$ by 1.

Now suppose there are $k_{p-1}$ points in (the modified) $S$ below $(p-1,n_{p-1})$. We then repeat the same process with $k_{p-1}$ outer jumps to again modify $S$ and further reduce the number of points in $b$ to $p-2$. Continue this way.
At the end of these sequence of moves, the number of points in $b$ will be empty and
the points in $c$ would be $(1,n_p),(2,n_p-1),\dots,(p,n_p-p),(p+1,n_q-1),\dots,(p+q,n_{q}-q)$ and this forces $n_q = n_p-p$.
If, at each stage there were $k_{p-i}$ points below $(p-i,n_{p-i})$, then, 
we have performed the inverse of the outer jump partition $\nu = (k_1,\dots,k_p)$, which clearly satisfies $k_1 \geq \cdots \geq k_p$.

We will now build the inner jump partition $\mu$.
Now consider the points in region $f$, which must have coordinates $(m_1,n_p-p-1),\dots, (m_s,n_p-p-s)$, where $s = r-p-q$ and $m_1 > \cdots > m_s$. We now repeat the same sequence of arguments as above, where each interchange corresponds to exchanging these points with points in $S$ to their left (and above). These correspond to inner jumps. If there are $(\ell_1,\dots,\ell_s)$ points to left of points with abscissas $(m_1,\dots,m_s)$, then these interchanges correspond to the inverse of the inner jump partition $\mu = (\ell_s,\dots,\ell_1)$. 

At the end of these sequence of moves, all points of $R$ will be consecutively decreasing in the region $c$, and hence we obtain the permutation $\pi = \pi_2$ above with the jump partition being $J = (\mu,\nu)$.
Notice that the interchanges in regions $b$ and $f$ can be performed independently.
If, initially all the points were in regions $\{g,h,j\}$, we would have had $\pi = \pi_1$. 
\end{proof}

\subsection{Proof of Theorem~\ref{thm:two-col}}
\label{sec:complete-proof}
From Theorem~\ref{thm:jumpseqgivesperm}, it follows that every jump partition of size $\Delta$ acting on a minimal permutation 
yields a permutation in $W_\lambda^\Delta$ and
Theorem~\ref{thm:Jpigivesuniqueperm} guarantees that no two distinct jump partitions give rise to the same permutation. Finally, Theorem~\ref{thm:permgivesuniquejpimin}  ensures that every permutation in $W_\lambda^\Delta$ can be obtained in a unique way starting from a jump partition acting on a minimal permutation. The proof of Theorem~\ref{thm:two-col} follows. As a consequence, we obtain the following structure theorem about such permutations.

\begin{corollary}
\label{cor:fixed-pts}
For any permutation $\pi$ in $W_\lambda^\Delta$ with $\lambda' = (s,r)$ and $\Delta < r$, one of the following is true.
\begin{itemize}
\item $\pi_i = s+1-i$ for $s-\Delta$ consecutive numbers starting at some position $i$, where $1 \leq i \leq \Delta+1$ and
$\pi_i = 2s+r+1-i$ for $r-\Delta$ consecutive numbers starting at some position $i$, where $s+1 \leq i \leq s+\Delta+1$.
\item $\pi_i = r+1-i$ for $r-\Delta$ consecutive numbers starting at some position $i$, where $1 \leq i \leq \Delta+1$ and
$\pi_i = s+2r+1-i$ for $s-\Delta$ consecutive numbers starting at some position $i$, where $r+1 \leq i \leq r+\Delta+1$.
\end{itemize}
\end{corollary}

\begin{proof}
From Theorem~\ref{thm:permgivesuniquejpimin}, there is a unique permutation $\pi_0$, which is either $\pi_1$ or $\pi_2$ in \eqref{minperms-twocol}, and a unique jump partition $J = (\mu,\nu)$ such that $\pi = J(\pi_0)$. The two cases correspond to these minimal permutations. We will argue the first case in which $\pi_0 = \pi_1$; the second one is similar.

Using the terminology of Lemma~\ref{lem:innerouterdisjoint}, the elements in the first block which are not fixed are 
$\{1,\dots,\ell(\mu),s-\ell(\nu)+1,\dots,s\}$. But since $\ell(\mu) + \ell(\nu) < \Delta < r \leq s$, there are at least $s-\Delta$ consecutive elements in $\pi_0$ which are undisturbed. A similar argument works for the second block
Finally, recall that $(\pi_0)_i = s+1-i$ for $1 \leq i \leq s$ and $(\pi_0)_i = 2s+r+1-i$ for $s+1 \leq i \leq s+r$.
\end{proof}

\section*{Acknowledgements}
We thank Amitabh Basu and Jim Haglund for discussions.
We also thank the anonymous referee for many useful comments.
The first author acknowledges support from the UGC Centre for Advanced Studies. and from DST grant DST/INT/SWD/VR/P-01/2014.
The second author was supported by NSF grant DMS-1261010 and a Sloan Research Fellowship. She would like to thank Nati Linial for discussions about this work in its initial stages.

\bibliographystyle{alpha}
\bibliography{minimal}

\end{document}